\documentclass[10pt,leqno,oneside]{amsart}

\usepackage[english]{babel}
\usepackage{amssymb}
\usepackage{amsmath}
\usepackage{graphicx}
\usepackage{amscd}
\usepackage{url}
\usepackage{vmargin}
\usepackage[all]{xy} 
\usepackage{url}
 \usepackage{mathrsfs}
 \usepackage{color}
 
\newtheorem{theorem}{Theorem}
\newtheorem*{thm}{Theorem}

\newtheorem{definition}[theorem]{Definition}
\newtheorem{temp-def}{Temporary definition}

\newtheorem{lemma}[theorem]{Lemma}
\newtheorem{lemma-def}[theorem]{Lemma-Definition}

\newtheorem{remark}[theorem]{Remark}

\def\proof{\smallskip\goodbreak{\it Proof.~~--~\kern.3em}
     \ignorespaces}%
\def\qedbox{$\square$}%
\def\qed{\ifmmode\qedbox\else\unskip\ \hglue0mm\hfill
     \qedbox\smallskip\goodbreak\fi}%
\newcommand{\CC}{\mathbb{C}}

\newcommand{\NN}{\mathbb{N}}

\newcommand{\Cc}{\mathscr{C}}

\newcommand{\Lc}{\mathscr{L}}

\newcommand{\RM}[1]{ \text{\sc \romannumeral #1 }}

\begin{document}
\title{Galois groupoid and confluence of difference equations}

\author{Guy Casale}
\address{Guy Casale, Univ Rennes, CNRS, IRMAR-UMR 6625, F-35000 Rennes, France 
}
\email{\tt guy.casale@univ-rennes1.fr}

\author{Damien Davy}
\address{Damien Davy, Univ Rennes, CNRS, IRMAR-UMR 6625, F-35000 Rennes, France
}
\email{\tt brandavy@outlook.fr}



\maketitle

\begin{abstract}
In this article we compute Galois groupoid of discret Painlev\'e equations. Our main tool  is a semi-continuity theorem for the Galois groupoid in a confluence situation of a diffrence equation to a differential equation.
\end{abstract}

\section{Introduction}
The aim of this article is to determine Galois groupoids of some discrete dynamical systems $\Phi : M \dasharrow M$. The phase space $M$ is an algebraic varaiety and $\Phi$ is a rational dominant map.  Galois groupoid is the ``differential algebraic'' envelope of $\Phi$, it describes the richest algebraic geometric structure (see \cite{gromov}) invariant by $\Phi$.

For linear differential equation, E. Picard  \cite{picard}, E. Vessiot \cite{vessiot} and later E.R. Kolchin \cite{kolchin} developed a Galois theory. To an order $n$ linear differential equation with coefficient in the differential field $\left( \mathbb C(x), \frac{\partial}{\partial x} \right)$ is associated an algebraic sub-group of $\rm{GL}_n(\mathbb C)$.  This group measures algebraic relations between a basis of solutions and theirs derivatives. This theory interacted with many different areas of mathematics see \cite{bertrandBBK} or \cite{singerVDP}. A discrete analog was proposed by  Franke \cite{franke} for linear difference equations (see also \cite{qsingerVDP}).

In \cite{malgrange}, a similar object was defined by B. Magrange for foliations, including nonlinear differential equations. Malgrange called it the \textit{Galois groupoid} of the foliation. In the present article, we present an extension of this definition to different dynamical systems.
Galois groupoid of a rational dominant map is defined and first examples are studied in \cite{casalediscret} and \cite{granier}. Let us give a quick, unformal and non complete definition, (see section \ref{def} for the formalisation).

Let $B$ be a smooth irreducible algebraic variety with en rational dominant self map $\sigma : B \dasharrow B$. It corresponds to a difference field of rational functions $(\CC(B), \sigma^\ast)$. Let $M$ be a smooth irrreducible algebraic varierty with a projection $\pi : M \to B$ and $\Phi : M \dasharrow M$ a rational dominant lift of $\sigma$. The difference equation associated to these datas is the equation $y\circ \sigma = \Phi \circ y$ on local holomorphic sections of $\pi$, $y : U \to M$, $U \subset B$. Galois groupoid of $\Phi$ (or its difference equation) is denoted by $Mal(\Phi/B)$, roughly speaking it is the \emph{algebraic pseudogroup} of local holomorphic transformations of fibers of $\pi$ \emph{generated by restrictions of $\Phi$ to fibers}.  By considering all order $k$ Taylor expansions of all elements in $Mal(\Phi/B)$ at all theirs points of definition one gets a finite dimensional algebraic variety called $Mal_k(\Phi/B)$. The Galois groupoid has a natural structure of pro-algebraic variety.

\subsection{Confluence and specialisation}

When a dynamical system depends on parameters, we wonder how the Galois groupoid depends on these parameters.
For linear differential equation an answer was given by \cite{goldman}. A generalisation of these specialisation results including difference equation can be found by a Tannakian approach in \cite{andre}. As a byproduct, these theorems can be used to obtained bound on the Galois group of a differential equation given as the continuous limite of a difference equation. 

Let $B \dasharrow S$ be a rational map invariant by $\sigma$. We can choose a value $s \in S$ and considere the restriction of the dynamical system $\Phi_s : M_s \dasharrow M_s$ above $\sigma_s : B_s \dasharrow B_s$. The general opinion is that the dynamical behaviour of $\Phi_s$ for a particular value of $s\in S$ must be simpler that the behaviour on the generic fiber. This is the specialisation theorem.

\begin{thm}[\ref{specialisationthm} page \pageref{specialisationthm}] In the situation above, 
$\dim Mal_k(\Phi_s/B_s) \leq \dim_S Mal_k(\Phi/B)$ with equality for general $s\in S$.
\end{thm}

A property is true for a general $s \in S$ if it is true for $s$ out of a countable union of proper algebraic subvarieties of $S$.

The second theorem concerns vector fields obtained as continuous limits of discrete dynamical systems.  This means that $S$ is one dimensional  and near a special point $s_0\in S$, one can write $\Phi_s = Id + (s-s_0)X +o(s-s_0)$ ; $Id$ is the identity map on $M_{s_0}$ and $X$ is a rational vector field on $M_{s_0}$.  Again we will prove that the dynamical behaviour of $X$ is simpler that the behaviour of $\Phi$. This is the confluence theorem.
  
\begin{thm}[\ref{confluencethm} page \pageref{confluencethm}] In the situation above, 
$\dim Mal_k(X/B_{s_0}) \leq \dim_S Mal_k(\Phi/B)$.
\end{thm}

\subsection{The main example}

A discrete Painlev\'e equation is a birational self-map $\Phi$ of $\CC^2 \times \Cc$ fibered above an automorphism without periodic points of the curve $\Cc$ of the independent variable $n$. This curve may be non compact. Formally the difference equation of invariant curves paramaterized by $n$ is the difference equation. Following \cite{grammaticos}, two properties are imposed to be called ``discrete-Painlev\'e equation'' :
\begin{enumerate}
\item It must have the singularities confinement property. There exists a fiberwise extension $M$ of $\CC^2 \times \Cc\to \Cc$ such that $\Phi$ can be extended as a biholomorphism  of $M$.
\item It must degenerate on a differential Painlev\'e equation. There exists a deformation $\widetilde{M}$ of $M$ above $(\CC, 0)$ and a deformation $\widetilde{\Phi}$ of $\Phi$ such that for $\epsilon \in (\CC^\ast,0)$  the couple $(\widetilde{M}_\epsilon, \widetilde{\Phi}_\epsilon)$ is birrational to $ (M,\Phi)$ and at $0$ one has a Taylor expansion $\widetilde{\Phi}_\epsilon = Id + \epsilon X +o(\epsilon)$ with $X$ a vector field on $\widetilde{M}_0$ whose trajectories parameterized by open set in $\Cc_0$ are Painlev\'e transcendents. This vector field is called a continuous limit of $\Phi$.    
\end{enumerate}
In \cite{sakai}, H. Sakai achieved the classification of total spaces $M$ of definition of discrete Painlev\'e equations. For a overview on difference Painlev\'e equations the reader may see \cite{noumi}. Let us describe the easiest non trivial example.

\subsubsection{The discrete Painlev\'e 2 equation}
Considere the rational dominant map :
$$
\begin{array}{rrcl}
\Phi_{\RM{2}}: &\CC^6 &\dasharrow & \CC^6 \\
                & \begin{bmatrix}n\\x\\y\\a\\b\\c\end{bmatrix} & \mapsto &  \begin{bmatrix}n+1\\-y + \frac{(a+bn)x +c}{1-x^2}\\x\\a\\b\\c\end{bmatrix}
\end{array}.
$$

Invariant analytic curves parameterized by $n$ are solutions of the so-called discrete Painlev\'e two equation :
$$
dP_{\RM{2}}(a,b,c) : x(n+1) + x(n-1) =  \frac{(a+bn)x(n) +c}{1-x(n)^2}.
$$
This equation degenerates on the second Painlev\'e equation. For $\epsilon \in \CC^\ast$, the change of variable $t= n\epsilon$, $f = y/\epsilon$, $g =(x-y)/\epsilon^2$, $\alpha = (a-2)/\epsilon^4$, $\beta =(b-\epsilon^2)/\epsilon^4$ and $\gamma = c/\epsilon^3$ can be used to get a trivial family $\Phi_{II}(\epsilon)$ of rational dominant maps above $\CC^\ast$ with a special degenerated fiber at $0$. A direct computation gives the expression of $\Phi_{\RM{2}}$ in these new coordinates
$$
\begin{array}{rrcl}
\Phi_{\RM{2}}(\epsilon): &\CC^6 &\dasharrow & \CC^6 \\
                & \begin{bmatrix}t\\g\\f\\ \alpha \\ \beta \\ \gamma\end{bmatrix} & \mapsto &   \begin{bmatrix}t\\g\\f\\ \alpha \\ \beta \\ \gamma\end{bmatrix} + \epsilon \begin{bmatrix} 1\\2f^3 + tf +\gamma\\ g \\ 0\\ 0\\ 0 \end{bmatrix} +o(\epsilon)
\end{array}.
$$
from which one gets the confluence of the family $\Phi_{\RM{2}}(\alpha,\beta,\gamma,\epsilon)$ on the vector field $ X_\gamma = \frac{\partial}{\partial t} +  g\frac{\partial}{\partial f} +(2f^3 + tf +\gamma) \frac{\partial}{\partial g}$ whose trajectories parameterized by $t$ are solutions of the second Painlev\'e equation $P_{\RM{2}}(\gamma)$ : $\frac{d^2f}{dt^2} = tf +2f^3+ \gamma$.

\subsubsection{The second Painlev\'e equations}

From \cite{casaleweil}, we know the Galois groupoid of $X_0$ and from \cite{davy} we know that that Galois groupoid of $X_\gamma$ over $\CC$ is the pseudogroup of invariance of the times form $dt$ and the closed 2-form $\iota_{X_\gamma}dt \wedge df \wedge dg$ for very general values of $\gamma$. 
The main theorems, {\bf Theorem \ref{specialisationthm} }and {\bf Theorem \ref{confluencethm}}, can be used to determine for very general values of $(a,b,c)$ the Galois groupoid of $\Phi_{\RM{2}}(a,b,c)$ over $\CC(n)$: 

\begin{thm}[\ref{painlevéthm} page \pageref{painlevéthm}]
 \label{gros}
 $$
Mal(\Phi_{\RM{2}}(a,b,c)/\CC(n))= \{ \varphi : (\CC^2,p_1) \to (\CC^2,p_2)\  |\  p_1,p_2 \in \CC^3,  \varphi^\ast dx\wedge dy = dx\wedge dy\}
$$
\end{thm}

\subsection{Galois groupoid and irreducibility}

The notion of irreducibility of a differential equation is as old as differential equations. It is formalized by K. Nishioka \cite{nishiokaP1} and Umemura \cite{umemurairred1,umemurairred2} following original ideas from Painlev\'e's Stochkolm Lessons \cite{painleve}. A discrete version of irreducibility can be found in \cite{nishiokadecomposable}. Let us recall a weaker version of the definition.
\begin{definition}
A rational second order difference equation $E : x(n+2) = F(n,x(n),x(n+1))$ is reducible if there exist a tower of field extension 
$$
\CC(n) =K_0 \subset K_1 \ldots \subset K_N 
$$ 
such that
\begin{enumerate}
\item each intermediate extension $K_{i-1} \subset K_i$ is of one of the following type
\begin{enumerate}
\item algebraic,
\item generated by entries of a fundamental solution of a linear system with coeffiecients in $K_{i-1}$,
\item generated by a solution of a first order non linear difference equation with coefficients in $K_{i-1}$,
\end{enumerate} 

\item there exists a solution $x \in K_N$ with $n, x(n), x(n+1) $ algebraically independent in $K_N$.
\end{enumerate} 
\end{definition}

From numerous papers by S. Nishioka \cite{nishiokaD7,nishiokaA7, nishiokaA6}, we know the irreducibility of many discrete Painlev\'e equation. The proofs are done by direct and very specific computation.

The aim of our work is to provide a new proof of irreducibility of most of the discrete Painlev\'e equations based on the computation of theirs Galois groupoids using the confluence of such discrete equations on differential Painlev\'e equations. Such a proof should prove ireducibility of any discrete dynamical system which admit as contiuous limite second order equation with a big enough Galois groupoid (such as \ref{gros}). This seems to be done using theorems of this article together with discrete analog of results from \cite{casaleirred}. This will be done in future work.

\section{Fibered dynamical systems and transversal Galois groupoid}
\label{def}

When a dynamical system $\Phi : M \dasharrow M$ comes from a difference equation, it preserves the foliation by fibers of the map $M \to B$ given by the independant coordinates. Moreover in the case with parameters, $\Phi$ preserves the fibers of the map given by the parameters.  Let  $M$ be the phase space, $B$ the space of the independant variable and parameters and $S$ the parameters space. We have a commutative diagram
$$ 
\begin{array}{ccc}
M &\overset{\Phi}{\dasharrow} &M \\
\downarrow & & \downarrow\\
B &\overset{\sigma}{\dasharrow} &B \\
\downarrow & & \downarrow\\
S & = & S
\end{array}$$
where $\sigma$ is the map corresponding to the operator involved in the difference equation.

\subsection{The fibered frames bundle}
Let $q$ be the dimension of fibers of $M$ over $B$. The space of frames on fibers of $M \to B$ is
$$R(M/B) = \{r : (\CC^q,0) \to M_b \ | \ b \in B\text{ and }\det(Jr) \not = 0 \}$$
where $Jr$ is the Jacobian matrix of $r$ at $0$.
Its coordinates ring of $R(M/B)$ is
$$\left(Sym(\CC[M]\otimes \CC[\partial_1,\ldots,\partial_q]) / \Lc \right) [1/jac]$$ where 

\begin{itemize}
  \item the tensor product is a tensor product of $\CC$-vector spaces;
  \item $Sym( V )$ is the $\CC$-algebra generated by the vector space $V$;
  \item $Sym(\CC[M]\otimes \CC[\partial_1,\ldots, \partial_q])$ has a structure of $\CC[\partial_1,\ldots,\partial_q]$-differential algebra {\it via} 
          the right composition of differential operators;
  \item the Leibniz ideal $\Lc$ is the $\CC[\partial_1,\ldots, \partial_q]$-ideal generated by $fg\otimes 1 - (f\otimes1)(g\otimes1)$ for all $(f,g) \in \CC[M]^2$, $h\otimes \partial_i$ for $h\in \CC[B]$ and $1\leq i \leq q$ and $1\otimes 1 - 1$;
  \item the quotient is then localized by $jac$ the sheaf of ideals (not differential !) generated by $\det\left(  [x_i \otimes \partial_j]_{i,j}\right)$
          for a transcendental basis $(x_1, \ldots, x_q)$ of $\CC(M)$ over $\CC(B)$ on Zariski open subset of $M$ where such a basis is defined. 
 \end{itemize}

Local coordinates $(b_1, \ldots b_p, x_1, \ldots, x_q)$ on $M$ such that the projection on $B$ is the projection on $b$'s coordinates, induce local coordinates on $R(M/B)$ {\it via} the Taylor expansion of maps $r$ at $0$ :
$$
r(\epsilon_1\ldots, \epsilon_q) = \left(b^0_1,\ldots b^0_p, \sum r_1^{\alpha} \frac{\epsilon^\alpha}{\alpha!}, \ldots,\sum r_q^{\alpha} \frac{\epsilon^\alpha}{\alpha!}   \right).
$$ 
One denotes $x_i^{\alpha} : R(M/B) \to \CC$ the function defined by $x_i^{\alpha}(r) = r_i^{\alpha}$. 
This function is the element $x_i\otimes \partial^\alpha$ in 
$\CC[R(M/B)]$. 
\begin{enumerate}
\item The action of $\partial_j : \CC[R(M/B)] \to \CC[R(M/B)
]$ can be written in
local coordinates and gives the total derivative operator $\sum_{i,\alpha} x_i^{\alpha + 1_j} \frac{\partial}{\partial x_i^{\alpha}}$ where $1_j$ is the multiindex whose only non zero entry is the $j^{th}$ and its values is 1.
\item The vector space $\CC[\partial_1,\ldots,\partial_q]$ is filtered by $\CC[\partial_1,\ldots,\partial_q]^{\leq k}$ the spaces of operators of order less than $k$. This gives a
filtration of $\CC[R(M/B)]$ by  $\CC$-algebras of finite type. 
\item These algebras are coordinate ring of the space of $k$-jet of frames $R_k(M/B) =\{j_k r \ |\ r \in R(M/B)\}$.
\item The action of $\partial_1,\ldots,\partial_q$ have degree $+1$ with respect to the filtration.

\end{enumerate}

\subsection{Prolongation of dominant morphism and vector fields}

 Morphisms (resp. derivations) from $\CC[M]$ to $\CC[M]$ with a non zero Jacobian determinant and  preserving $\CC[B]$ act on $\CC[R(M/B)]$ as morphisms (resp. derivations).
 
  If the map $\Phi$ is regular and induces $\Phi^\ast :\CC[M] \to \CC[M] $. Its action of the frame bundle is defined by $Sym(\Phi \otimes 1)$ on $Sym(\CC[M]\otimes \CC[\partial_1,\ldots, \partial_q])$, it can be easily checked that the Leibniz ideal is preserved. The induced map $\CC[R(M/B)] \to \CC[R(M/B)]$ corresponds to a endomorphism of $R(M/B)$ denoted by $R\Phi$ and is called the prolongation of the morphism. The prolongation of a derivation of $\CC[M]$ preserving $\CC[B]$ is a derivation of $\CC[R(M/B)]$ defined in the same way.

If the map $\Phi$ is rational and dominant with domain of definition $M^\circ$, its prolongation on $R(M^\circ/B)$ is a rational dominant map on $R(M/B)$. Rational vector fields on $M$ can be prolonged on $R(M/B)$ in a similar way.\\

  Prolongations of morphisms and derivations of $\CC[M]$ on $\CC[R(M/B)]$ have degree $0$ with respect to the filtration defined above. Prolongations commute with the differential structure. When $X$ is a rational vector field on $M$ preserving $M \to B$, its prolongation $RX$ on $R(M/B)$ can be computed with a explicit formula :
  
Let $X = \sum c_i(b)\frac{\partial}{\partial b_i} + \sum a_i(x,b)\frac{\partial}{\partial x_i}$ be a vector field preserving $M\to B$ in local coordinates. One gets 
$RX = \sum c_i(b)\frac{\partial}{\partial b_i} + \sum_{i,\alpha} \partial^\alpha(a_i)\frac{\partial}{\partial x_i^\alpha}.$

\subsection{The space $R(M/B)$ is a principal bundle over $M$.}

 Let us describe this structure here.
The pro-algebraic group $$\Gamma = \{ \gamma : \widehat{(\CC^q,0)} \overset{\sim}{\rightarrow} \widehat{(\CC^q,0)}; \text{ formal invertible}\}$$ is the projective limit of groups $\Gamma_k = \{ j_k \gamma | \gamma \in \Gamma\}$. It
 acts on $R(M/B)$ by composition and  $ R(M/B) \times \Gamma  \to R(M/B) \underset{M}{\times} R(M/B)$ sending $(r,\gamma)$ to $(r,r\circ \gamma)$ is an isomorphism.
The action of $\gamma \in \Gamma$ on $R(M/B)$ is denoted by $S\gamma : R(M/B) \to R(M/B)$ as it acts as a change of
source coordinate of frames. At the coordinate ring level, this action is given by the action of formal change of coordinate on 
$\CC[\partial_1,\ldots,\partial_q]$ followed by the evaluation at $0$ in order to get operators with constant coefficients.  
   This action has degree $0$ with respect to the filtration induced by the order of differential operators meaning that for any $k$, the 
bundle of order $k$ frames $R_kM$ is a principal bundle over $M$ for the group $\Gamma_k = \{ j_k\gamma \ | \ \gamma \in \Gamma \}$.

\subsection{Compatibility of the differential and principal structures}

Let $\epsilon_1, \ldots \epsilon_q$ be coordinates on $(\CC^q,0)$ such that $\partial_i = \frac{\partial}{\partial \epsilon_i}$.
The Lie algebra of $\Gamma$ is the Lie algebra of formal vector fields vanishing at $0$ : $\widehat{\mathfrak{X}}^0(\CC^q,0)$. The infinitesimal action of $v = \sum a_i(\epsilon) \partial_i \in \widehat{\mathfrak{X}}(\CC^q,0)$, the Lie algebra of formal vector fields on $R(M/B)$ is  $f \otimes P \mapsto f \otimes (P\circ v)|_{\epsilon = 0}$. When $v$ belongs to $\bigoplus \CC \partial_i$, this action is the differential structure ; when $v$ belongs to $\widehat{\mathfrak{X}}^0(\CC^q,0)$, it is the action induced by the action of the group $\Gamma$.

\subsection{The fibers automorphism groupoid}
A groupoid is naturally associated to any principal bundle. In the case of $R(M/B)$ one can give three definitions of this groupoid.
\begin{enumerate} 
\item The groupoid $Aut(M/B)$ is the groupoid of formal invertible maps:
$$
\{ \widehat{\varphi} : (M_b,p_1) \overset{\sim}{\rightarrow} (M_{b'},p_2)\ |\ (b,b') \in B\times B , p_1\in M_b , p_2 \in M_{b'}\}
$$
\item 
The map 
$$
\begin{array}{rcl}
R(M/B) \times R(M/B) & \rightarrow& Aut(M/B) \\
(r,s) &\mapsto& r\circ s^{\circ-1}
\end{array}
$$
realized the quotient by the diagonal action of $\Gamma$.
\item The groupoid of $\Gamma$-equivariant maps between fibers of $\pi : R(M/B) \to M$ is $Aut(M/B)$.
\end{enumerate}

The second definition is adapted to prove that $Aut(M/B)$ is pro-algebraic groupoid.

\subsection{Galois groupoids}

It is a proper generalisation of the Galois group of linear differential or difference equation.
It is a quotient of the Galois groupoid defined in \cite{malgrange} for foliations and in \cite{granier, casalediscret} for rational dominant maps.

\subsubsection{Galois groupoid of a rational vector field}

The Galois groupoid for vector fields where already studied in \cite{davy}. We recall the definition in the fibered situation. The proof are immediate from \cite{davy} or \cite{casaledavy} and we left them to the reader.

A rational functions $H \in \CC(R(M/B))$ 
such that $RX \cdot H = 0$ is a differential invariant of $X$. Let $Inv(X) \subset \CC(R(M/B))$ be the subfield of differential invariants of $X$. 
Let $W$ be a model for $Inv(X)$ and $\pi : R(M/B) \dasharrow W$ be the dominant map induced by the inclusion $Inv(X) \subset \CC(R(M/B))$.

\begin{definition}
The transversal Galois groupoid is $Mal(X/B) = R(M/B) \underset{W}{\times}R(M/B) \subset Aut(M/B)$
\end{definition}

To defined properly this fiber product one needs to restrict $\pi : (R(M/S))^o \to W$ on its domain of definition then  $(R(M/B) \underset{W}{\times} R(M/B))$ is defined to be the Zariski closure of  $(R(M/B)^o \underset{W}{\times} R(M/B)^o)$ in $R(M/B) \times R(M/B)$.

 When $B = \{\ast\}$, D. Davy \cite{davy} proved that this definition is equivalent to Malgrange's original definition. In particular Malgrange shows in  \cite{malgrange} that  there exists a Zariski open subset $M^{o}$ of $M$ such that the restriction of $Mal(X)$ to $Aut(M^o)$ is a subgroupoid. 

%

\subsubsection{Galois groupoid of a rational dominant map}
A rational functions $H \in \CC(R(M/B))$ 
such that $H \circ R\Phi = H$ is a (fibered) differential invariant of $\Phi$. Let $Inv(\Phi) \subset \CC(R(M/B))$ be the subfield of differential invariants of $\Phi$.
 
Let $W$ be a model for $Inv$ and $\pi : R(M/B) \dasharrow W$ be the dominant map induced by the inclusion $Inv \subset \CC(R(M/B))$. 

\begin{definition}
The (fibered) Galois groupoid is $Mal(\Phi /B) = R(M/B) \underset{W}{\times}R(M/B) \subset Aut(M/B)$
\end{definition}

Here again, to defined properly this fiber product one needs to restrict $\pi : (R(M/S))^o \to W$ on its domain of definition then  $(R(M/B) \underset{W}{\times} R(M/B))$ is defined to be the Zariski closure of  $(R(M/B)^o \underset{W}{\times} R(M/B)^o)$ in $R(M/B) \times R(M/B)$. 

The following lemmas are important in the proof of the specialisation theorem.

\begin{lemma}
Let $O\subset R(M/B)^o$ be a Zariski open subset then $\overline{O\underset{W}{\times}O} =  \overline{R(M/B)^o \underset{W}{\times} R(M/B)^o}$.
\end{lemma}

\begin{lemma}
Galois groupoid of $\Phi$ is the Zariski closure of the set of Taylor expansions of iterates of $\Phi$.
\end{lemma}

\begin{proof} Let $k$ be an integer and consider the order $k$ frame bundle $R_k(M/B)$.
From \cite{amerikcampana}, the dominant map $\pi : R_k(M/B) \dasharrow W_k$ induced by the inclusion $Inv_k \subset \CC(R_k(M/B))$ as the following property : for a general $j_k(r) \in R_k(M/B)$ the Zariski closure of the fiber $\pi^{-1}(\pi( j_k(r)))$ is the Zariski closure of the $R_k(\Phi)$ orbit of $j_k(r)$. 

Let $q : R_k(M/B) \times R_k(M/B) \to Aut_k(M/B)$ be the quotient by the diagonal action of $\Gamma$. The image $q(j_k(\Phi \circ r), j_k(r))$ is $j_k(\Phi)$ : the order $k$ Taylor expansion of $\Phi$ at $r(0)$. Then $Mal_k(\Phi) = q(R_k(M/B) \underset{W_k}{\times}R_k(M/B))$ contains Taylor expansions of $\Phi$.

For general $m \in M$, one can find a frame $r$ with $r(0)=m$ such that $\pi^{-1}(\pi( j_k(r))) \times \{r\} \subset R_k(M/B)\times \{r\}$ is the Zariski closure of the orbit of $\{r\}\times\{r\}$ for $\Phi$ acting on the first factor. The projection $q(\pi^{-1}(\pi( j_k(r))))$ 
is the Zariski closure of the set of the Taylor expansions of $\Phi^{\circ n}$ at $m$. It is also the subset of $Mal_k(\Phi/B)$ of order $k$ jet of maps with source at $m$.

Let $T_k$ be the Zariski closure of the set of all the Taylor expansions of iterates of $\Phi$. The subvarieties $T_k$ and $Mal_k(\Phi)$ coincide for source out of a closed subvariety of $M$. By minimality $Mal_k(\Phi) \subset T_k$. For $n \in \NN$ let $j_k\Phi^{\circ n} : M \to Aut_k(M/B)$ be the map sendind $m$ to the Taylor expansion of $\Phi^{\circ n}$ at $m$. This map is rational on $M$ and belongs to $Mal_k(\Phi/B)$ for general values of $m$ thus for any $m\in M$ (where the map is defined).
This proves the equality.
\end{proof}

\begin{remark} When $\Phi$ is the map arising from a linear difference equation on a vector bundle $M \to B$ : 
\begin{enumerate}
\item the pseudogroup $Mal(\Phi/B)$ is ``the analog'' of the intrinsec Galois group of the equation over the difference field $(\CC(B), \sigma)$
\item the pseudogroup $Mal(\Phi/ \ast)$ is ``the analog'' of the  intrinsec differential Galois group of the difference equation over $\CC(B)$ with difference operator $\sigma$ and differential structure given by the exterior differential $d : \CC(B) \to \CC(B) \underset{\CC[B]}{\otimes} \Omega^1_B$ 
\end{enumerate}
\end{remark}
%
%

\section{A specialisation theorem and a confluence theorem}\label{confluence}

\subsection{The specialisation theorem}

When the equation depends on parameters, the base $B$ is fibered on $S$ the parameter space and the map $\sigma : B \dasharrow B$ preserves the fibers of this map.
The first result is a comparison theorem of Galois groupoid of $\Phi$ with Galois groupoid of its restriction on a fiber $\Phi_s : M_s \dasharrow M_s$ for $s \in S$

\begin{theorem}\label{specialisationthm}
For all $k \in \NN$, $\dim Mal_k(\Phi_s/B_s) \leq \dim_S Mal_k(\Phi/B)$ with equality for general $s\in S$
\end{theorem}

\begin{proof}
The dimension of $Mal_k(\Phi/B)$ is the dimension of $M$ plus the dimension of the Zariski closure of a general orbit of $R_k(\Phi)$ in $R_k(M/B)$. One has to prove that the dimension of the Zariski closure of a particular orbit is smaller that the general one.
The proof is done in \cite{amerikcampana}. It can be adapted from \cite{bonnet} where this inequality is proved for leaves of algebraic foliations. 

For $j_k r\in R_k(M/B)$, $V(j_k r)$ is the Zariski closure of the orbit of $j_k r$. We will compare Hilbert polynomial of these subvarieties and deduce the wanted inequalities of dimensions.

The ideal of   $V(j_k r)$ is $ I(j_k r) = \{f \in \CC[R_k(M/B)] | \forall n \in \NN \  j_k(r)  \in \overline{\{f \circ \Phi^{\circ n} =0\}} \}$ and the subvector space of equation of degree less than or equal to $d$ is also described by these linear equations in $f$ depending on $r$. Let $h_{j_k r}(d)$ be the dimension of this space. There exists a Zariski open subset $O \subset R_k(M/B)$ and an integer $h_k$ such that for any $r$,  $h_{j_k r}(d) \geq h_k$ with equality for $j_k r \in O$.

This implies that the Hilbert polynomial of $V(j_k r)$ is contant for $j_k r$ general in $R_k(M/B)$, and that for other frames, it is smaller. The dimension of a subvariety is the degree of its Hilbert polynomial. The theorem is proved.
\end{proof}

\subsection{The confluence theorem}
In this section, we will assume the parameter space to be 1-dimensional with coordinate $s$.
This theorem compares  Galois groupoids of $\Phi$ with Galois groupoid of $X$ a vector field on $M_{s_0}$ such that $\Phi_s = Id + (s-s_0) X +o(s-s_0)$

\begin{theorem}\label{confluencethm}
For all $k \in \NN$, $\dim Mal_k(X/B_{s_0}) \leq \dim_S Mal_k(\Phi/B) $
\end{theorem}

\begin{proof}
First, the restrictions of differential invariants of $\Phi$ at $s=s_0$ give differential invariants of $X$. By expansion around $s_0$, $R\Phi_s = Id + (s-s_0)RX+ o(s-s_0)$ and $H = H^0+o(s-s_0)$ then $H\circ R\Phi_s = H$ implies $RX \cdot H^0 = 0$. This means that the fiber of $Mal_k(\Phi/B)$ at $s_0$ contains $Mal_k(X/B_{s_0})$ or equivalently $Inv(R_k\Phi)|_{s_0}$ contains $Inv(R_kX)$.

To prove the inequality one has to prove that $\dim Mal_k(\Phi/B)_{s_0}\leq \dim_S Mal_k(\Phi/B)$. As the codimension of these varieties are given by the transcendence degree of theirs fields of differential invariant, we will compare $Inv(R_k\Phi)_{s_0}$ and $Inv(R_k\Phi)$.

Let $H_1, \ldots H_p$ be a transcendence basis of $Inv(R_k\Phi)$ over $\CC(s)$. Because $s$ is a differential invariant, one can assume that the restriction of the basis to $s=s_0$ are well defined rational functions $H_1^{0}, \ldots, H_p^0$ and the vanishing order $\ell$ of $dH_1\wedge \ldots \wedge dH_p$ at $s=s_0$ is minimal. If $\ell =0$, these restrictions are functionally independant, the theorem is proved. If not, let $P(H_1^{0}, \ldots, H_p^0)$ be the minimal polynomial of $H_p^0$ over $\CC[H_1^{0}, \ldots, H_{p-1}^0]$. The invariant $P(H_1, \ldots, H_p)$ can be written as $(s-s_0) F$ with $F$ a function of $s, H_1, \ldots, H_p$. Up to some factorisation by $s-s_0$ one can assume that there exist a index $i$ such that $(\partial_i P)^0$ is well defined and not zero.  One has, up to some sign and mod $ds$,
$$
d(P(H_1,\ldots H_p)) \wedge dH_1\wedge \ldots \widehat{dH_i}\ldots \wedge dH_p = \partial_i P dH_1\wedge \ldots \wedge dH_p = (s-s_0)dF \wedge dH_1\wedge \ldots \widehat{dH_i}\ldots \wedge dH_p
$$
This contradicts the minimality of $\ell$.
\end{proof}

\subsection{Corollaries about Galois groupoids of discrete Painlev\'e equations}

A discrete Pain\-le\-v\'e equation is a birational map $\Phi$ of $\CC^2 \times \Cc$ fibered above a automorphism without periodic points of the curve of the independent variable $n$. This curve may be non compact. Formally the difference equation of invariant curves paramaterized by $n$ is the difference equation. Two properties are imposed to be called ``discrete-Painlev\'e equation'' :
\begin{enumerate}
\item It must have the singularity confinement property. There exists a fibers compactification $M$ of $\CC^2 \times \Cc\to \Cc$ such that $\Phi$ can be extended as a biholomorphism  of $M$.
\item It must degenerate on a differential Painlev\'e equation. There exists a deformation $\widetilde{M}$ of $M$ above $\CC$ and a deformation $\widetilde{\Phi}$ of $\Phi$ such that for $\epsilon \in \CC^\ast$  $(\widetilde{M}_\epsilon, \widetilde{\Phi}_\epsilon)$ is birrational to $ (M,\Phi)$ and at $0$ one has the Taylor expansion $\widetilde{\Phi}_\epsilon = Id + \epsilon X +o(\epsilon)$ with $X$ a vector field on $\widetilde{M}_0$ whose trajectories parameterized by open set in $\Cc$ are Painlev\'e transcendents. This vector field is called a continuous limit of $\Phi$.    
\end{enumerate}
The classification of the phase spaces of these equations can be found in \cite{sakai}. The equations themselves can be found in \cite{noumi}.  The spaces are classified by families indexes by affine Weyl groups $W$, each family depends on finite number of parameters $a$. Then Painlev\'e equations are denoted by $dP_W(a)$, $dP^\ast_W(a)$, $dP^{alt}_W(a)$, $qP_W(a)$ \ldots In this article any of these equations will be denoted by $dP_W(a)$. 
In Sakai's classification \cite{sakai}, it is proved that there exists an invariant relative 2-form $\omega \in \Omega^2_{M/ \Cc}$. In the framework of this article this means that we have an inclusion :
$$
Mal(dP_W(a) / \Cc) \subset  \{ \varphi : (M_{n},p) \to (M_{n'},q) |  (n,n')\in \Cc^2,  \varphi^\ast \omega =  \omega\}.
$$ 
This inclusion, the second property above and Cartan generic local classification of algebraic pseudogroups \cite{cartan} enable us to compute Galois groupoid of discrete Painlev\'e equations.  
\begin{theorem}\label{painlevéthm}
For any affine Weyl group $W$ and for general values of parameters $a$
$$
Mal(dP_W(a) / \Cc) =  \{ \varphi : (M_{n},p) \to (M_{n'},q) |  (n,n')\in \Cc^2,  \varphi^\ast \omega =  \omega\}.
$$ 
\end{theorem}

\begin{proof}
First one will determined the codimension $2$ subvariety  given for $n \in \Cc$ by $M_{n,n} = Mal(dP_W(a) / \Cc) \cap   \{ \varphi : (M_{n},p) \to (M_{n},q),\}$. It is a subgroupoid of 
$Inv(\omega_n) =  \{ \varphi : (M_{n},p) \to (M_{n},q) |  \varphi^\ast \omega =  \omega\}
$. The restriction of these pseudogroups on a fiber $M_n$ give algebraic pseudogroups on two dimensional manifold. Such objects have been classified by Elie Cartan \cite{cartan} up to analytic change of coordinates near a generic point. 
From the specialisation theorem we know that $\dim (M_{n,n})_k$, the algebraic variety of jet of order $k$ of element of $M_{n,n}$, has a quadratic growth in $k$. From Cartan results we know that there is only one quadratic growth subgroupoid of  $Inv(\omega_n)$ : it is  $Inv(\omega_n)$ itself. 

Now the two groupoids dominate $\Cc \times \Cc$ and have same fibers above the diagonal, they are equals. 
\end{proof}
%

\bibliographystyle{alpha}
\bibliography{references}

 \end{document}